\documentclass[12pt]{amsart}

\usepackage{amsfonts,amsmath,amssymb,amscd,amsthm}
\usepackage{mathrsfs}
\usepackage{faktor}    
\usepackage[utf8]{inputenc}
\usepackage[english]{babel}

\usepackage[T1]{fontenc}
\usepackage[eulergreek]{mathastext}

\usepackage{bbm}
\usepackage[all]{xy}
\usepackage{pst-node}
\usepackage{enumitem}
\usepackage{tikz-cd}
\usepackage{tikz}
\usetikzlibrary{arrows, automata, positioning}
\usepackage{pdfpages}
\usepackage[left=40mm, top=20mm, right=40mm, bottom=25mm]{geometry} 
\usepackage{tabto}
\TabPositions{2 cm, 4 cm, 6 cm, 8 cm}
\usepackage[unicode, pdftex]{hyperref}
\hypersetup{
  colorlinks   = true, 
  urlcolor     = blue, 
  linkcolor    = blue, 
  citecolor   = red 
}

\DeclareMathOperator{\Ima}{Im}

\DeclareMathOperator{\op}{op}
\DeclareMathOperator{\HS}{HS}
\DeclareMathOperator{\Hamm}{Hamm}
\DeclareMathOperator{\Frob}{Frob}

\DeclareMathOperator{\rank}{rank}

\DeclareMathOperator{\Tr}{Tr}

\DeclareMathOperator{\Sym}{Sym}
\DeclareMathOperator{\Sp}{Sp}

\DeclareMathOperator{\U}{U}
\DeclareMathOperator{\M}{M}

\newcommand{\N}{{\mathbb{N}}}
\newcommand{\Z}{{\mathbb{Z}}}
\newcommand{\C}{{\mathbb{C}}}

\newcommand{\scG}{\mathscr{G}}

\newtheorem{theorem}{Theorem}[section]
\newtheorem{corollary}[theorem]{Corollary}
\newtheorem{lemma}[theorem]{Lemma}

\newtheorem{proposition}[theorem]{Proposition}

\newtheorem*{lemma*}{Lemma}
\newtheorem*{proposition*}{Proposition}
\newtheorem*{theorem*}{Theorem}
\newtheorem*{corollary*}{Corollary}
\newtheorem*{claim*}{Claim}

\theoremstyle{definition}
\newtheorem{definition}[theorem]{Definition}

\newtheorem*{definition*}{Definition}

\theoremstyle{remark}
\newtheorem{remark}[theorem]{Remark}

\theoremstyle{definition}

\newtheorem{conjecture}[theorem]{Conjecture}

\newcommand\setItemnumber[1]{\setcounter{enumi}{\numexpr#1-1\relax}}

\title[]{On $L^1$-approximation of groups}
\author{Benjamin Bachner, Alon Dogon and Alexander Lubotzky}

\begin{document}

\maketitle
    
\begin{abstract}
A longstanding open problem in the intersection of group theory and operator algebras is whether all groups are MF, that is, approximated by asymptotic representations with respect to the operator norm.
More generally, in his ICM address from 2018, Thom asked whether there exist groups that are not approximated with respect to the Schatten $p$-norm for $1 \leq p \leq \infty$.
The cases of $1 < p < \infty$ were addressed in previous works of De Chiffre et al. and Lubotzky-Oppenheim.
Here we settle the case $p=1$.
  %
\end{abstract}

\section{Introduction}

Let $\mathscr G = \{ (G_n, d_n)\}$ be a family of groups $G_n$ equipped with bi-invariant metrics $d_n$.
A group $\Gamma$ is called \emph{$\mathscr G$-approximated} if there exist maps (not necessarily homomorphisms) $\varphi_n: \Gamma \to G_n$, such that:
\begin{enumerate}
\item The sequence $\varphi_n$  is an \emph{asymptotic homomorphism}, namely $\forall g,h \in \Gamma,\;\lim_{n\to \infty}d_n(\varphi_n(gh), \varphi_n(g)\varphi_n(h))  = 0$.
\item $\varphi_n$ \emph{asymptotically separates} $\Gamma$, that is, 
$$\forall g \in \Gamma \setminus \{1\}, \; \liminf_n \;d_n(\varphi_n(g), 1_{G_n}) > 0.$$
\end{enumerate}
The group $\Gamma$ is said to be \emph{$\scG$-stable} if for every asymptotic homomorphism $\varphi_n: \Gamma \to G_n$ (as in $(1)$ above), there exists \emph{true} homomorphisms $\psi_n: \Gamma \to G_n$ such that 
$$\forall g \in \Gamma, \; \lim_{n \to \infty} d_n(\varphi_n(g), \psi_n(g)) = 0.$$

The notions of approximation and stability have been studied quite intensively in recent years (cf. \cite{ThomICM, DGLT, LO, CL1, CL2, Dog} and the references therein).

The main example of families of metric groups $\scG$ studied are:
\begin{enumerate}
    \item $\scG_{\Hamm} = \{ (\Sym(n), d^{\text{Hamm}}_n) \}_n$, where for $\sigma, \tau \in \Sym(n)$, 
    $$ d^{\Hamm}_n(\sigma, \tau) = \frac{1}{n}|\{i \in [n] | \; \sigma(i) \neq \tau(i) \}|.$$
\end{enumerate}

In all the instances considered in this paper, we have $G_n = U(n)$, and the metric $d_n$ will be defined by a unitarily invariant norm $\| \cdot \|_n$ on the matrices $M_n(\C)$ via $d_n(u,v) = \| u-v \|_n$ for $u,v \in U(n)$. For $A \in M_n(\C)$, denote $|A| = \sqrt{A^*A}$, and now:
\begin{enumerate}
    \setItemnumber{2}
    \item $\scG_{\Frob}$: The \emph{Frobenius norm} $ \| A\|_{\Frob} = \Tr{(|A|^2)}^{1/2}$.
    \setItemnumber{3}
    \item $\scG_{\HS}$: The (normalized) \emph{Hilbert--Schmidt norm}, or simply \emph{HS-norm}, $ \| A \|_{\HS} = \frac{1}{\sqrt{n}} \| A\|_{\Frob}$.
    \setItemnumber{4}
    \item $\scG_p$: For $1 \leq p < \infty$, The \emph{Schatten $p$-norm} defined by $\| A\|_p = \Tr(|A|^p)^{1/p}$ (For $p=2$, this is the Frobenius norm).
    \setItemnumber{5}
    \item $\scG_{\op}$: The \emph{operator norm} norm $\| A \|_{\op} = \sup_{v \in \C^n, \|v\|=1} \| Av\|$. This is considered as the case $p = \infty$ and we write $\|A\|_{\infty} = \|A \|_{\op}$.
\end{enumerate}
Finding groups which are not $\scG$-approximated is notoriously hard. 
One of the motivations to study stability comes from the following observation (see \cite{GR, AP} for $\scG_{\Hamm}$, but valid in all the cases listed above).

\begin{proposition}
    \label{prop:Glebsky_Rivera}
    Let $\scG$ be a class as above, where $G_n$ are linear groups, and $\Gamma$ a finitely generated group.
    If $\Gamma$ is $\scG$-approximated and $\scG$-stable, then it is residually finite.
\end{proposition}
This implies that if one found a finitely generated group $\Gamma$ which is $\scG$-stable and not residually finite, then it is \emph{not} $\scG$-approximated.
This line of proof was implemented first in \cite{DGLT} to give examples of non-Frobenius approximated groups, i.e. groups which are not $\mathscr G_2$-approximated.
Following this breakthrough, Thom asked in his ICM address \cite[Question 3.11]{ThomICM} about all $1 \leq p \leq \infty$.
In response, Lubotzky and Oppenheim \cite{LO} solved almost all cases and gave non-$\mathscr G_p$-approximated groups for $1 < p < \infty$.
These examples are the so called "$\ell$-adic Deligne central extensions", see \cite[Section 5.2]{DGLT} for their construction.
There are thus four cases left to handle: $p=1$, the Hamming distance on permutation groups, the HS-norm and the operator norm.

We prove that $\ell$-adic Deligne central extensions are not $\mathscr G_1$-approximated,
solving the case $p=1$ in Thom's question.
Interestingly, our method is different from the previous works -- using a short argument, we manage to establish that these groups are not $\mathscr G_1$-approximated, without the need to show that they are $\mathscr G_1$-stable. The question of $\mathscr G_1$-stability for these groups remains open.

With that said, this is only modest progress as the last three cases are the holy-grail of the subject: Groups which are $\scG_{\Hamm}$-approximated are called \emph{sofic}, the ones which are $\scG_{\HS}$-approximated are called \emph{hyperlinear} or \emph{Connes embeddable}, and the ones which are $\scG_{\op}$-approximated are called \emph{MF}\footnote{There are multiple, possibly inequivalent definitions for MF groups in the literature, this is the one introduced in  \cite{CDE}, see \cite{Schaf} for the more C$^*$-algebraic one.}.
Whether all groups are sofic, hyperlinear or MF are three major open problems in group theory. 
Given the previous works, it is natural to study whether the $\ell$-adic Deligne central extensions are the much desired \emph{concrete} counter examples to these three problems.
The reader is referred to \cite{CL1, CL2, GT, CDL} for efforts to prove they are not sofic, and to \cite{Dog} for efforts to prove that \emph{real} Deligne central extensions are not hyperlinear.

The novelty in our approach is the use of at least two \emph{different} norms on $\U(n)$ to deduce inapproximability.
Usually, when two different bi-invariant metrics $d_n\leq  d_n'$ are given on $G_n$, it is not clear that $\{G_n,d_n' \}$-approximation (respectively, $\{G_n,d_n'\}$-stability) implies $\{G_n,d_n \}$-approximation (respectively, $\{G_n,d_n\}$-stability), nor vice versa.

Denote by $\ker_{\mathrm{RF}}(\Gamma)$ the kernel of the map from $\Gamma$ to its profinite completion $\widehat{\Gamma}$. 
In other words, $\ker_{\mathrm{RF}}(\Gamma)$ consists of elements in $\Gamma$ that are trivial in all finite quotients.

\begin{definition}
    A finitely generated group $\Gamma$ is said to be \emph{of Deligne-type} if there is a central element $J \in \Gamma$ of order $2$ such that $J \in \ker_{RF}(\Gamma)$.
\end{definition}

See \cite{Deligne}  and \cite[Section 5]{DGLT} for examples, and many more in \cite{stover2024residual,Hill}.
Our main theorem is a version of Proposition \ref{prop:Glebsky_Rivera} which deals with two \emph{different} Schatten norms, that is applicable to groups of Deligne-type. This refines the previously taken route of proving stability in order to deduce inapproximability.

\begin{theorem}\label{thm:main}
    Let $\Gamma$ be a group of Deligne-type, and let $1 \leq p \leq q \leq \infty$. 
    If $\Gamma$ is $\mathscr G_q$-stable, then it is not $\mathscr G_{p}$-approximated.
    In particular, if $\Gamma$ is Frobenius stable, then it is not $\scG_p$-approximated for any $1 \leq p \leq 2$.
\end{theorem}

As a consequence of the construction of Deligne-type groups which are Frobenius stable obtained in \cite[Proposition 5.6 and Theorem 5.1]{DGLT}, we have:

\begin{corollary}
    There exist finitely presented groups which are not $\mathscr G_1$-approximated.
\end{corollary}

    In fact, by the more recent work of \cite[Theorem 1.3]{bader2023stability} and the original work of Deligne \cite{Deligne}, it also follows that the degree $4$-cover of $\Sp_{2g}(\Z)$, for $g \geq 3$, is not $\mathscr G_1$-approximated.
    

The interplay between two different matrix norms, one dominated by the other, is also useful in the context of the normalized Hilbert--Schmidt norm and the operator norm.
We again obtain a refinement to Proposition \ref{prop:Glebsky_Rivera}, which serves as a potential approach to tackle the MF question.

\begin{proposition}\label{prop:mixed_stability_implies_not_MF}
    Let $\Gamma$ be a group of Deligne type. Assume that the following holds:
    
     $(*)$ \emph{For every $\| \cdot \|_{\op}$-asymptotic homomorphism $\varphi_n: \Gamma \to \U(k_n)$, there exists a sequence of homomorphisms $\psi_n: \Gamma \to \U(k_n)$ such that for all $x \in \Gamma$, $\lim_{n} \| \varphi_n(x) - \psi_n(x) \|_{\HS} = 0$. }
    
    Then $\Gamma$ is not MF.
\end{proposition}

We will refer to condition $(*)$ above as \emph{operator-HS-stability} (as in \cite{eckhardt2025residually}). 

Using a slightly more refined argument, we obtain the following monotonicity phenomena for approximability, which is of independent interest.

\begin{proposition}\label{prop:monotonicity_for_hyp_MF}
    Let $\Gamma$ be a finitely generated group with a finite normal subgroup $N$ such that $\Gamma / N$ is hyperlinear.
    If $\Gamma$ is MF, then it is hyperlinear.
\end{proposition}

\begin{proposition}\label{prop:monotonity_for_p-norm}
    Let $1 \leq p \leq q \leq \infty$. 
    Let $\Gamma$ be a finitely generated group with a finite normal subgroup $N$ such that $\Gamma / N$ is $\mathscr G_q$-approximated.
    If $\Gamma$ is $\mathscr G_p$-approximated, then it is $\mathscr G_{q}$-approximated.
\end{proposition}

Let us end with suggesting the following conjecture:

\begin{conjecture}\label{conj:main}
Let $\Gamma$ be a finitely generated group such that $H^2(\Gamma, V) = 0$ for every unitary representation of $\Gamma$ on a Hilbert space $V$, then $\Gamma$ is operator-HS-stable.
\end{conjecture}
%

Recall that the $\ell$-adic Deligne central extensions of \cite{DGLT} satisfy the above cohomological vanishing condition (in fact, they have property $[T_2]$ of Bader and Sauer \cite{arXiv:2308.06517}), hence if Conjecture \ref{conj:main}
 is true, these groups are not MF.

\section{Preliminaries}

\subsection{Unitarily invariant matrix norms }

Throughout this subsection, for every $d\in \N$, we fix \emph{unitarily invariant norm} $\| \cdot \|$ on $\M_d(\C)$, meaning $\| uxv \|= \|x\|$ for all $u,v \in \U(d), x \in \M_d(\C)$.
Note that all the matrix norms mentioned in the introduction are instances of unitarily invariant norms.
 Let us recall some basic properties at this level of generality.

\begin{lemma}[Proposition 1.3 in \cite{DGLT}]\label{lem:uni_inv_norm_properties}
Let $A,B,C\in {\rm M}_d(\mathbb C)$, it holds that:
\begin{enumerate}
\item[$(1)$] $\Vert ABC\Vert\leq \|A\|_{\op}\Vert B\Vert \|C\|_{\op},$
\item[$(2)$] $\lVert A\rVert=\lVert A^*\rVert=\lVert |A|\rVert$,
\item[$(3)$] If $A$ and $B$ are positive semi-definite matrices and $A\leq B$, then $\lVert A\rVert\leq \lVert B\rVert.$
\end{enumerate}
\end{lemma}

The following lemma guarantees that any unitary which is close to being an involution is close to a  unitary which is also an involution.
This plays a key role in the proof of Theorem \ref{thm:main}.

\begin{lemma} [Proposition 1.4 in \cite{DGLT}]\label{lem:almost_involution}
Let $A\in{\rm U}(d)$. Then there is a unitary $B\in{\U}(d)$ such that $B^2=1$ and
\[\lVert B-A\lVert\leq \lVert 1 -A^2\rVert,\]
for all unitarily invariant norms.
\end{lemma}

We will also need the following folklore lemma.

\begin{lemma}\label{lem:almost_unitary}
    For every $n \in \N$ and $M \in \M_n(\C)$, there exists $R \in \U(n)$ such that $\| M - R \| \leq \| M^*M - 1 \|$ for all unitarily invariant norms. 
\end{lemma}

The above Lemma is stated in the literature for various unitarily invariant norms separately. 
We note that it in fact has a uniform proof using the singular value decomposition that can be found in \cite[Lemma 2.2]{AD}, by simply observing that the proof carries over for arbitrary unitarily invariant norms.
%
%
%
    
Using this, we can prove the following regarding restriction of $\| \cdot \|_p$-asymptotic homomorphisms to almost invariant subspaces.

\begin{proposition}\label{prop:inv_sub_asym_hom}
Let $1 \leq p \leq \infty$, and let $\varphi_n: \Gamma \to \U(d_n)$ be a $\Vert \cdot \Vert_p$-asymptotic homomorphism. Assume that there exist nonzero projections $P_n \in \M_n(\C)$ such that $\| \varphi_n(g)P_n - P_n \varphi_n(g) \|_p \to 0$ for all $g \in \Gamma$.
Then there exists a $\Vert \cdot \|_p$-asymptotic homomorphism $\psi_n: \Gamma \to \U(\Ima(P_n))$ such that:
$$ \lim_n \| \psi_n(g) - P_n \varphi_n(g) P_n\|_p = 0 \text{ for all $g \in \Gamma$.}$$
\end{proposition}
\begin{proof}
    Define $\widetilde{\varphi}_n(g) = P_n \varphi_n(g) P_n \in \M_{\rank(P_n)}(\C)$. 
    Observe that we have for all $g,h \in \Gamma$:
    \begin{align*}
        \| \widetilde{\varphi}_n(g) &\widetilde{\varphi}_n(h) - \widetilde{\varphi}_n(gh) \|_p = \| P_n \varphi_n(g) P_n \varphi_n(h) P_n - P_n \varphi_n(gh) P_n \|_p \\
        & \leq \| P_n \varphi_n(g)  \varphi_n(h) P_n - P_n \varphi_n(gh) P_n \|_p +  \|P_n [\varphi_n(g) , P_n ] \varphi_n(h) P_n \|_p \\
        &\leq \| \varphi_n(g)  \varphi_n(h) - \varphi_n(gh) \|_p + \| [\varphi_n(g) , P_n ] \|_p.
    \end{align*}
    Note that we used Lemma \ref{lem:uni_inv_norm_properties} and that $\| P_n \|_{\op} =1$.
    Consequently, we have $\lim_n \| \widetilde{\varphi}_n(g) \widetilde{\varphi}_n(h) - \widetilde{\varphi}_n(gh) \|_p = 0$ for all $g,h \in \Gamma$.
    Since $\varphi_n(g)$ is unitary, this also shows $\lim_n \| \widetilde{\varphi}_n(g)^* \widetilde{\varphi}_n(g) - P_n \|_p = 0$.
    By Lemma \ref{lem:almost_unitary}, we can replace $\widetilde{\varphi}_n(g)$ by unitaries $\psi_n(g) \in \U(\Ima(P_n))$ which are asymptotically close in $\| \cdot \|_p$ to $\widetilde{\varphi}_n(g)$, and $\psi_n: \Gamma \to \U(\Ima(P_n))$ will be the desired asymptotic homomorphism.
    \end{proof}

\section{Proofs}

We start by proving the main result\footnote{We thank Andreas Thom for suggesting a shortcut for our original proof.}.

\begin{proof}[Proof of Theorem \ref{thm:main}]

Let $\Gamma$ be a group of Deligne type, and $J \in \ker_{RF}(\Gamma)$ a central element of order $2$.
Let $1 \leq p \leq q \leq \infty$, and assume that $\Gamma$ is $\mathscr G_q$-stable.
Assume by contradiction that $\Gamma$ is $\mathscr G_p$-approximated.
Let $\varphi_n: \Gamma \to \U(d_n)$ be a $\| \cdot \|_p$-asymptotic homomorphism which is $\| \cdot \|_p$-separating.
By passing to a subsequence, we may assume that there exists $\delta > 0$ such that $\| \varphi_n(J) - 1 \|_p \geq  \delta$ for all $n \in \N$.
Since $J^2 = 1$, we have $\lim_n \| \varphi_n(J)^2 - 1 \|_p =0$. 
Consequently, by Lemma \ref{lem:almost_involution} there exist $U_n \in \U(d_n)$ with $U_n^2 =1$ such that $\lim_n \| U_n - \varphi_n(J) \|_p = 0$. 
As such, For all $n$ large enough we have $\| U_n -1 \|_p \geq \delta/2$, in particular, $U_n \neq 1$.
Since $U_n^2=1$, this implies that it must have an eigenvector with eigenvalue $-1$, which implies $\| U_n - 1 \|_q \geq \| U_n - 1\|_{\op} \geq 2$.

 %
 Now, since $ \| \cdot \|_p \geq \| \cdot \|_q$, we conclude that $\varphi_n$ is also a $\| \cdot \|_q$-asymptotic homomorphism, and that $\| \varphi_n(J) - 1 \|_q \geq 1$ for $n$ large enough.
 Since $\Gamma$ is $\scG_{q}$-stable, there exists a sequence of genuine representations $\rho_n: \Gamma \to \U(d_n)$ such that $\lim_n \| \rho_n(g) - \varphi_n(g) \|_{q} = 0$ for all $g \in \Gamma$, in particular for $g = J$.
 This implies that $\rho_n(J) \neq 1$ for $n$ large enough.
 Thus, by Malcev's theorem, we can  find a finite quotient of $\Gamma$ in which $J$ is not trivial, contradicting $J \in \ker_{RF}(\Gamma)$.
\end{proof}

\begin{remark}
    In the above proof the fact that $J$ is central is never used, so Theorem \ref{thm:main} actually holds for any group $\Gamma$ that has an element $J \in \ker_{\mathrm{RF}}(\Gamma)$ of order $2$. However, the centrality is used in the proof of Proposition \ref{prop:mixed_stability_implies_not_MF} below.
\end{remark}

\begin{proof}[Proof of Proposition \ref{prop:mixed_stability_implies_not_MF}]

Let $\Gamma$ be a group of Deligne type, and $J \in \ker_{RF}(\Gamma)$ a central element of order $2$.
We show that if $\Gamma$ is operator-HS-stable, then it is not MF.
Assume by contradiction that there exists a $\| \cdot \|_{\op}$-asymptotic homomorphism  $\varphi_n: \Gamma \to \U(d_n)$ for which there exists $\delta > 0$ such that $\| \varphi_n(J) - 1 \|_{\op} \geq  \delta$ for all $n \in \N$.
As before, by Lemma \ref{lem:almost_involution} there exist $U_n \in \U(d_n)$ with $U_n^2 =1$ such that $\lim_n \| U_n - \varphi_n(J) \|_{\op} = 0$.
Thus, we may replace $\varphi_n(J)$ by $U_n$ and assume without loss of generality that $\varphi_n(J)^2 = 1$ for all $n$.
Let $P_n = (\varphi_n(J) -1)/2$ be the projection onto the $-1$-eigenspace of $\varphi_n(J)$, and note that $P_n$ is a nonzero projection for all $n$ large enough.

Since $J$ is central and $\varphi_n$ is an $\|\cdot\|_{\op}$-asymptotic homomorphism, we have $\lim_n \| [\varphi_n(g), \varphi_n(J)] \|_{\op} = 0$ for all $g \in \Gamma$.
As such, $\lim_n \| [\varphi_n(g), P_n] \|_{\op} = 0$, 
so that $\varphi_n(g)$ all asymptotically preserve the subspace $\Ima(P_n)$.
Thus, by Proposition \ref{prop:inv_sub_asym_hom}, there exists a $\|\cdot\|_{\op}$-asymptotic homomorphism $\psi_n: \Gamma \to \U(\Ima(P_n))$ such that $\lim _n\| \psi_n(g) - P_n \varphi_n(g) P_n \|_{\op} = 0$ for all $g \in \Gamma$.
In particular, $\| \psi_n(J) + 1 \|_{\HS} \leq \| \psi_n(J)  - (-1) \|_{\op}  = \| \psi_n(J) - P_n \varphi_n(J) P_n\|_{\op} \to 0 $.
 Since $\Gamma$ is operator-HS-stable, there exists a sequence of genuine representations $\rho_n: \Gamma \to \U(\Ima(P_n))$ such that $\lim_n \| \rho_n(g) - \psi_n(g) \|_{\HS} = 0$ for all $g \in \Gamma$.
 All together this implies that $\rho_n(J) \neq 1$ for $n$ large enough.
 Thus, by Malcev's theorem, we can  find a finite quotient of $\Gamma$ in which $J$ is not trivial, contradicting $J \in \ker_{RF}(\Gamma)$.
\end{proof}

It is worthwhile to note that the argument above, which allows to construct an asymptotic homomorphism $\varphi_n: \Gamma \to \U(d_n)$ with the additional property $\varphi_n(J) = -1$, is analogous to \cite[Corollary 2.5]{CDL} in the setting of sofic approximations.
We now prove Proposition \ref{prop:monotonicity_for_hyp_MF} and leave the proof of Proposition \ref{prop:monotonity_for_p-norm} to the reader, as it has the same proof once $\| \cdot \|_{\op}$ is replaced by $\| \cdot \|_{p}$ and $\| \cdot \|_{\HS}$ is replaced by $\| \cdot \|_q$.


\begin{proof}[Proof of Proposition \ref{prop:monotonicity_for_hyp_MF}]
We first prove the proposition under the assumption that $N$ is central in $\Gamma$.
Assume $\Gamma / N$ is hyperlinear, $\Gamma$ is MF and let $\varphi_n: \Gamma \to \U(d_n)$ be a $\| \cdot \|_{\op}$-asymptotic homomorphism which is $\| \cdot \|_{\op}$-separating.
To show $\Gamma$ is hyperlinear, it is enough to show that every $1 \neq g_0 \in N$ can be separated by $ \| \cdot \|_{\HS}$-asymptotic homomorphisms. Indeed, all elements in $\Gamma \setminus N$ can be separated by $\| \cdot \|_{\HS}$-asymptotic homomorphisms of $\Gamma / N$ by assumption. As such, one can take the direct sum of finitely many $\| \cdot \|_{\HS}$-asymptotic homomorphisms separating each nontrivial element of $N$ together with an $\| \cdot \|_{\HS}$-asymptotic homomorphism separating elements in  $\Gamma \setminus N$, which will yield a $\| \cdot \|_{\HS}$-separating asymptotic homomorphism of $\Gamma$.

Since the finite group $N$ is $\mathscr{G}_{\op}$-stable (see for example \cite{de2019operator}), there exists a sequence of homomorphisms $\alpha_n: N \to U(d_n)$ which are asymptotically close to the restriction of $\varphi_n$ to $N$ with respect to $\| \cdot \|_{\op}$.
Thus, we may replace the restriction of $\varphi_n$ to $N$ by $\alpha_n$, and assume without loss of generality that the restriction of $\varphi_n$ to $N$ is a homomorphism.
Since $\| \varphi_n(g_0) - 1\|_{\op}$ is bounded away from $0$,  we can find  characters $\chi_n \in \widehat{N}$ such that $\chi_n(g_0)$ is bounded away from $1$ and $P_{\chi_n}$, the projection onto the $\chi_n$-isotypic component of $\alpha_n$, is nontrivial.
Since $N$ is finite and central, it can be shown that $\lim_n \| [ P_{\chi_n}, \varphi_n(g)] \|_{\op} = 0$ for every $g \in \Gamma$, and as a result we can use Proposition \ref{prop:inv_sub_asym_hom} to find a $\| \cdot \|_{\op}$-asymptotic homomorphism $\psi_n: \Gamma \to U(\Ima  P_{\chi_n})$ such that $\| \psi_n(g_0) - \chi_n(g_0) \cdot I\|_{\op} \to 0$. 
Since $\| \cdot \|_{\HS} \leq \| \cdot \|_{\op}$, we see that $\psi_n$ is a $\| \cdot \|_{\HS}$-asymptotic homomorphism for which $\| \psi_n(g_0) - \chi_n(g_0) \cdot I\|_{\HS} \to 0$, in particular it $\| \cdot \|_{\HS}$-asymptotically separates $g_0$ and we are done.

In case $N \triangleleft \Gamma$ is not central, we can consider its centralizer $\Lambda \leq \Gamma$.
As $N$ is finite and normal, $\Lambda \leq \Gamma$ has finite index.
If $\Gamma$ is MF, so is $\Lambda$.
Note that $N \cap \Lambda \leq \Lambda $ is central and $\Lambda  / (N\cap \Lambda ) \simeq \Lambda N / N \leq \Gamma / N$ is hyperlinear as well. 
It follows from the previous proof that $\Lambda $ is hyperlinear.
As $\Lambda \leq \Gamma$ has finite index, it follows that $\Gamma$ is hyperlinear as well.
\end{proof}

\subsection*{Acknowledgments} 
We are grateful to Andreas Thom for several useful comments which helped improve the paper, including a shortcut in the proof of Theorem \ref{thm:main}. We also thank Forrest Glebe for helpful remarks.

\subsection*{Funding}
All three authors were supported by the European Research
Council (ERC) under the European Union’s Horizon 2020 program (N. 882751), and a research grant from the Center for New Scientists at the Weizmann Institute of Science. AD was also supported by a Clore scholars grant from the Clore Israel Foundation.

\bibliographystyle{amsplain}
\bibliography{ref}

\vspace{0.5cm}

\noindent{\textsc{Department of Mathematics, Weizmann Institute of Science, Israel}}

\noindent{\textit{Email address:} \texttt{benjamin.bachner@weizmann.ac.il}} \\
\noindent{\textit{Email address:} \texttt{alon.dogon@mail.huji.ac.il}} \\
\noindent{\textit{Email address:} \texttt{alex.lubotzky@mail.huji.ac.il.}} \\

\end{document}